\documentclass[11pt,reqno]{amsart}
\usepackage{amsmath, amsfonts, amsthm, amssymb, color, cite, soul}
\usepackage{mathrsfs}
\usepackage{dsfont}
\usepackage{upref}
\usepackage{indentfirst}
\usepackage{appendix}
\usepackage{enumerate}
\usepackage[bookmarks,colorlinks,citecolor=blue,linkcolor=red]{hyperref}
\usepackage{bm}
 \usepackage{graphicx}
\allowdisplaybreaks

\numberwithin{equation}{section}
\numberwithin{figure}{section}
\newtheorem{theorem}{Theorem}[section]

\newtheorem{remark}{Remark}[section]
\newtheorem{lemma}{Lemma}[section]

\author[Y.-Z. Chen]{Yazhou Chen}
\address{College of Mathematics and Physics, Beijing University of
Chemical Technology, Beijing 100029, China}
\email{chenyz@mail.buct.edu.cn}

\author[Q.-L. He]{Qiaolin He}
\address{School of Mathematics, Sichuan University, Sichuan 610065, China}
\email{qlhejenny@scu.edu.cn}

\author[B. Huang]{Bin Huang}
\address{College of Mathematics and Physics, Beijing University of
Chemical Technology, Beijing 100029, China}
\email{abinhuang@gmail.com}

\author[X.-D. Shi]{Xiaoding Shi}
\address{College of Mathematics and Physics, Beijing University of
Chemical Technology, Beijing 100029, China}
\email{shixd@mail.buct.edu.cn}

\title[Thermodynamic Compressible Diffuse Interface Model]
{Global Strong Solution to a Thermodynamic Compressible Diffuse Interface Model with Temperature Dependent Heat-conductivity in 1-D}


\date{\today}

\begin{document}
\begin{abstract}
In this paper, we investigate the wellposedness of the non-isentropic compressible Navier-Stokes/Allen-Cahn system with the heat-conductivity proportional to a positive power of the temperature.
This system describes the flow of a two-phase immiscible heat-conducting viscous compressible mixture. The phases are allowed to shrink or grow due to changes of density in the fluid and incorporates their transport with the current.
We established the global existence and uniqueness of strong solutions for this system in 1-D, which means no phase separation, vacuum, shock wave,  mass or heat or phase concentration will be developed in finite time, although the motion of the two-phase immiscible flow has large oscillations and the interaction between the hydrodynamic and phase-field effects is complex. Our result can be regarded as a natural generalization of the Kazhikhov-Shelukhin's result ([Kazhikhov-Shelukhin. J. Appl. Math. Mech. 41 (1977)]) for the compressible single-phase flow with constant heat conductivity to the  non-isentropic compressible immiscible two-phase flow with degenerate and nonlinear heat conductivity.
\end{abstract}

\maketitle

\section{Introduction}
\label{sec:The compressible phase-field model}
The immiscible two-phase flows appear widely in  many fields, such as thermal power engineering, nuclear energy engineering, cryogenic engineering and aerospace, etc. An important feature of immiscible two-phase flow is the coexistence of two fluids with different phase states or components and the existence of the interface.
The theoretical analysis of the immiscible two-phase flow is much more difficult than that of single-phase flow. A common treatment method is the so-called separated flow model, which holds that the concept and method of single-phase flow can be applied to each phase of the two-phase flow system,  and the interface properties and motion between the two phases are considered.

Understanding the geometry and distribution of the interface is very important for determining the flow of immiscible two-phase flow.
The treatment of such two-phase flow interface is derived from the idea of physicist J.D. Van der Waals  \cite{V1894}, who regarded the interface of immiscible two-phase flow as a region with a certain thickness. Mathematical models based on this idea are often called diffusion interface models, such as the famous Navier-Stokes/Allen-Cahn system, which can be used to study the immiscible two-phase flow, such as phase transformation, chemical reactions, etc. see\cite{AC1979}--\cite{HMR2012}, \cite{K2012}, \cite{LT-1998}, 
and the references therein. In these literatures, by introducing diffusion interface instead of sharp interface, the authors can  avoid the difficulty of processing interface boundary conditions.

Here is a brief review of the model. For compressible immiscible two-phase flow, taking any one of the volume particles in the flow, we assume $M_i$ the mass of the components in the representative material volume $V$, $\phi_i=\frac{\rho_i}{\rho}$ the mass concentration, $\rho_i=\frac{M_i}{V}$  the apparent mass density of the fluid $i~(i=1,2)$. The total density is given by $\rho=\rho_1+\rho_2$, and $\phi=\phi_1-\phi_2$. We call $\phi$ the difference of the two components for the fluid mixture. Obviously, $\phi$  describes the distribution of the interface. Therefore, on this basis,  Blesgen \cite{B1999},  Heida-M$\mathrm{\acute{a}}$lek-Rajagopal  \cite{HMR2012} coupled the Navier-Stokes system  describing the flow of the single fluid and the Allen-Cahn equation describing the change of the phase field (Allen-Cahn \cite{AC1979}),  the Navier-Stokes/Allen-Cahn system is proposed (hereinafter called as NSAC equations):
\begin{equation}\label{original NSAC}
\left\{\begin{array}{llll}
\displaystyle \rho_{t}+\textrm{div}(\rho \mathbf{u})=0,\\
\displaystyle (\rho \mathbf{u})_{t}+\mathrm{div}\big(\rho \mathbf{u}\otimes \mathbf{u}\big)=\mathrm{div}\mathbb{T},
  \\
\displaystyle(\rho\phi)_{t}+\mathrm{div}\big(\rho\phi \mathbf{u}\big)=-\mu,\\
\displaystyle\rho\mu=\rho\frac{\partial f}{\partial \phi}-\mathrm{div}\big(\rho\frac{\partial f}{\partial \nabla\phi}\big),\\
\displaystyle(\rho E)_{t}+\mathrm{div}(\rho E\mathbf{u})=\mathrm{div}\big(\mathbb{T}\mathbf{u}+\kappa(\theta)\nabla \theta-\mu\frac{\partial f}{\partial \nabla\phi}\big),
\end{array}\right.
\end{equation}
where $ \tilde{\mathbf{x}}\in \Omega \subset \mathds{R}^N $,  $N$ is the spatial dimension, $ t>0$ is the time. The unknown functions $\rho(\tilde{\mathbf{x}},t)$, $\mathbf{u}(\tilde{\mathbf{x}},t)$, $\phi(\tilde{\mathbf{x}},t)$, $\theta(\tilde{\mathbf{x}},t)$  denote the total density, the velocity, the difference of the two components for the fluid mixture, and the absolute temperature respectively. $\mu(\tilde{\mathbf{x}},t)$ is called the chemical potential of the fluid. $\epsilon>0$ is the thickness of the diffuse interface. $\mathrm{div}$ and $\nabla$ are the divergence operator and gradient operator respectively. The Cauchy stress-tensor is represented by
\begin{equation}\label{T}
\mathbb{T}=2\nu\mathbb{D}(\mathbf{u})+\lambda(\mathrm{div}\mathbf{u})\mathbb{I}-p\mathbb{I}-\rho\nabla\phi\otimes\frac{\partial f}{\partial\nabla\phi},
\end{equation}
where $f$ is  the  fluid-fluid interfacial free energy density,  and it has the following form (Lowengrub-Truskinovsky \cite{LT-1998},  Heida-M$\mathrm{\acute{a}}$lek-Rajagopal \cite{HMR2012}):
 \begin{equation}\label{free energy density}
 f(\rho,\phi,\nabla\phi)\overset{\text{def}}{=}\frac{1}{4\epsilon}(1-\phi^2)^2+\frac{\epsilon}{2\rho}|\nabla \phi|^2,
\end{equation}
 $\mathbb{I}$ is the unit matrix, $\mathbb{D}\mathbf{u}$ is the so-called deformation tensor
\begin{equation}\label{Du}
  \mathbb{D}\mathbf{u}=\frac{1}{2}\big(\nabla \mathbf{u} +\nabla^{\top} \mathbf{u}\big),
\end{equation}
here and hereafter, superscript $\top$  denotes the transpose and all vectors are column ones.  $\nu>0,\lambda>0$ are  viscosity coefficients, satisfying
\begin{equation}\label{nu}
 \nu>0,\ \ \lambda+\frac{2}{N}\nu\geq0.
\end{equation}
$\kappa(\theta)$ is the
heat conductivity satisfying
\begin{equation}\label{kappa}
 \kappa(\theta)=\tilde{\kappa}\theta^\beta,
\end{equation}
with constants $\tilde{\kappa}>0$ and $\beta>0$.
The  total energy density $\rho E$ is given by
\begin{equation}\label{total energy density}
 \rho E=\rho e+\rho f+\frac{1}{2}\rho\mathbf{u}^2,
\end{equation}
where $\rho e$ is the internal energy, $\frac{\rho\mathbf{u}^2}{2}$ is the kinetic energy,
$p=p(\rho,\theta)$, $e=e(\rho,\theta)$ and $f=f(\rho,\phi,\nabla\phi)$ obey the  second law of thermodynamics (Lions \cite{Lions1998}):
\begin{equation}\label{second law of thermodynamics}
  ds=\frac{1}{\theta}\big(d(e+f)+pd(\frac{1}{\rho})\big),
\end{equation}
where $s$ is the entropy. Then we deduce from \eqref{second law of thermodynamics}:
\begin{equation}\label{relation}
  \frac{\partial s}{\partial \theta}=\frac{1}{\theta}\frac{\partial (e+f)}{\partial \theta},\ \ \frac{\partial s}{\partial \rho}=\frac{1}{\theta}\big(\frac{\partial (e+f)}{\partial \rho}-\frac{p}{\rho^2}\big),
\end{equation}
 which implies the following compatibility equation
\begin{equation}\label{pressure}
p=\rho^2\frac{\partial (e+f)}{\partial \rho}+\theta\frac{\partial p}{\partial \theta}=\rho^2\frac{\partial e(\rho,\theta)}{\partial \rho}-\frac\epsilon2|\nabla\phi|^2+\theta\frac{\partial p}{\partial \theta}.
\end{equation}
Throughout the paper, we concentrate on ideal polytropic gas, that is, $p$ satisfies
\begin{equation}\label{pressure p}
p(\rho,\theta)=R\rho\theta-\frac\epsilon2|\nabla\phi|^2,
\end{equation}
and
 $e$ satisfies
\begin{equation}\label{e}
  e=c_v\theta+\mathrm{constant},
\end{equation}
where $c_v$ is the specific heat capacity. Substituting  \eqref{T}, \eqref{total energy density}, \eqref{free energy density}, \eqref{pressure} and \eqref{e} into \eqref{original NSAC}, then the NSAC system is simplified to
\begin{equation}\label{NSFAC3d}
\left\{\begin{array}{llll}
\displaystyle \rho_{t}+\textrm{div}(\rho \mathbf{u})=0,\\
\displaystyle \rho\mathbf{u}_{t}+\rho(\mathbf{u}\cdot\nabla)\mathbf{u}-2\tilde{\nu}\mathrm{div} \mathbb{D}\mathbf{u}-\tilde{\lambda}\nabla\mathrm{div}\mathbf{u}=-\mathrm{div}\big(\epsilon\nabla\phi\otimes\nabla\phi-\frac{\epsilon}{2}|\nabla\phi|^2+\theta\frac{\partial p}{\partial \theta}\big), \\
\displaystyle \rho\phi_{t}+\rho \mathbf{u}\cdot\nabla\phi=-\mu,\\
\displaystyle \rho\mu=\frac{\rho}{\epsilon}(\phi^3-\phi)-\epsilon \Delta\phi,\\
\displaystyle c_v\big(\rho\theta_{t}+\rho \mathbf{u}\cdot\nabla\theta\big)+\theta p_\theta\mathrm{div}\mathbf{u}-\mathrm{div}(\kappa(\theta)\nabla \theta)=2\tilde{\nu}|\mathbb{D}\mathbf{u}|^2+\tilde{\lambda}(\mathrm{div}\mathbf{u})^2+\mu^2.
\end{array}\right.
\end{equation}
In this paper, we consider the one-dimensional problem of the system \eqref{NSFAC3d}:
\begin{equation}\label{NSFAC-euler}
\left\{\begin{array}{llll}
\displaystyle \rho_{t}+(\rho u)_{\tilde{x}}=0,\\
\displaystyle \rho u_{t}+\rho uu_{\tilde{x}}+(R\rho\theta)_{\tilde{x}}=\nu u_{\tilde{x}\tilde{x}}-\frac\epsilon2(\phi_{\tilde{x}}^2)_{\tilde{x}}, \\
\displaystyle \rho\phi_{t}+\rho u\phi_{\tilde{x}}=-\mu,\\
\displaystyle \rho\mu=\frac{\rho}{\epsilon}(\phi^3-\phi)-\epsilon \phi_{\tilde{x}\tilde{x}},\\
\displaystyle c_v(\rho\partial_{t}\theta+\rho u\theta_{\tilde{x}})+R\rho\theta u_{\tilde{x}}-(\kappa(\theta)\theta_{\tilde{x}})_{\tilde{x}}=\nu u_{\tilde{x}}^2+\mu^2,
\end{array}\right.
\end{equation}
for $(\tilde{x},t)\in[0,1]\times[0,+\infty)$, where $\nu=2\tilde\nu+\tilde\lambda>0$.
The initial boundary conditions are following
\begin{equation}\label{NSFAC-euler-boundary}
  (u,\phi_{\tilde{x}},\theta_{\tilde{x}})(0,t)=(u,\phi_{\tilde{x}},\theta_{\tilde{x}})(1,t)=0,
\end{equation}
\begin{equation}\label{NSFAC-euler-initial}
 (\rho,u,\theta,\phi)\big|_{t=0}=(\rho_0,u_0,\theta_0,\phi_0), \ \ \tilde{x}\in(0,1).
\end{equation}
Without loss of generality, we assume that
\begin{equation}\label{rho}
  \int_0^1\rho_0(\tilde{x})d\tilde{x}=1,
\end{equation}
and
\begin{equation}\label{constant}
\nu=R=c_v=\tilde{\kappa}=1.
\end{equation}
In Lagrange coordinates
\begin{equation}\label{Lagrange coordinate}
 x=\int_0^{\tilde{x}}\rho(\xi,t)d\xi,
\end{equation}
the system \eqref{NSFAC-euler}--\eqref{NSFAC-euler-initial} can be rewritten as
\begin{equation}\label{NSFAC-Lagrange}
\left\{\begin{array}{llll}
\displaystyle v_t-u_x=0,\\
\displaystyle u_t+(\frac{\theta}{v})_x=(\frac{u_{x}}{v})_x-\frac\epsilon2(\frac{\phi_x^2}{v^2})_x, \\
\displaystyle \phi_t=-v\mu,\\
\displaystyle \mu=\frac{1}{\epsilon}(\phi^3-\phi)-\epsilon(\frac{\phi_x}{v})_x,\\
\displaystyle \theta_t+\frac{\theta}{v}u_x-(\frac{\theta^\beta\theta_x}{v})_x=\nu\frac{u_x^2}{v}+v\mu^2,
\end{array}\right.
\end{equation}
with the boundary condition
\begin{equation}\label{boundary condition}
 (u,\phi_x,\theta_x)(0,t)=(u,\phi_x,\theta_x)(1,t)=0, \quad t\geq0.
\end{equation}
and the initial value condition
\begin{equation}\label{initial condition}
 (v,u,\theta,\phi)(x,0)=(v_0,u_0,\theta_0,\phi_0)(x),  \qquad x\in(0,1),
\end{equation}
where
$$v=\frac1\rho.$$

There are a lot of works on the research of the global existence and large time behavior of solutions to the compressible heat-conducting Navier-Stokes system. Kazhikhov-Shelukhin \cite{KS1977} first proposed the global existence of solutions with large initial data. In the year since, significant progress has been made, see \cite{AZ1989}--\cite{AKM1990},  \cite{J1996}, \cite{J1998}, \cite{K1987} and the references therein. It should be noted that these works are given under the assumption that the heat conductivity is a positive constant.  If $\kappa$ depends on temperature, Kawohl \cite{K1985}, Jiang \cite{J1994-1,J1994-2}  and Wang \cite{W2003} established the global existence of smooth solutions for compressible heat-conducting Navier-Stokes system with boundary condition of $(u,\theta_x)(0,t)=(u,\theta_x)(1,t)=0$.  and the methods used there relies heavily on the non-degeneracy of  the heat conductivity $\kappa$  which  cannot be applied directly to  the degenerate and nonlinear case  ($\beta>0$).  Pan-Zhang\cite{PZ2015} generalize the above results to the degenerate case \eqref{kappa} where $\beta\in(0,\infty)$, and more improvement results associated with this degenerate case, please refer to Duan-Guo-Zhu \cite{DGZ2017}, Huang-Shi-Sun \cite{HSS2020}, Huang-Shi \cite{SH2019} and the references therein.

In terms of compressible two-phase flows, so far as we know, most of the work focuses on the isentropic case. Feireisl-Petzeltov$\acute{a}$-Rocca-Schimperna \cite{FPRS2010}, Chen-Wen-Zhu \cite{CWZ2019}  proved the global existence of the weak solution one after another, the method the used is based on the renormalization weak solution framework for the compressible Navier-Stokes system introduced by Lions \cite{Lions1998}. Kotschote \cite{K2012} showed the existence and uniqueness of local strong solutions for arbitrary initial data. Ding-Li-Luo \cite{DLL2013} established the existence and uniqueness of global strong solution in 1D for initial density without vacuum states. Chen-Guo \cite{CG2017} generalized  the result of Ding-Li-Luo \cite{DLL2013} to the case that the initial vacuum is allowed.

In this paper,  we   focus on  the non-isentropic compressible flow for two-phase immiscible mixture.
Our purpose is to study the existence and uniqueness of global strong solution for the non-isentropic NSAC systems  \eqref{NSFAC-Lagrange} even with large initial data.
More specifically, for general initial conditions without vacuum state, we study the global existence of the solution for the system  \eqref{NSFAC-Lagrange}--\eqref{initial condition}.
 Now we give our main result as following

\begin{theorem}\label{thm-global}
 Assume that
\begin{equation}\label{condition 1}
  (v_0,\theta_0)\in H^1(0,1),\ \  \phi_0\in H^2(0,1),\ \ u_0\in H_0^1(0,1),
\end{equation}
and
\begin{equation}\label{condition 2}
  \inf_{x\in(0,1)}v_0(x)>0,\ \ \inf_{x\in(0,1)}\theta_0(x)>0,\ \ \phi_0(x)\in[-1,1].
\end{equation}
Then,  the initial boundary value problem \eqref{NSFAC-Lagrange}-\eqref{initial condition} has a unique strong solution $(v,u,\theta,\phi)$ such that for fixed $T>0$, satisfying
\begin{equation}
\left\{\begin{array}{llll}
 \displaystyle v,\theta\in L^\infty(0,T;H^1(0,1)),\ \ u,\phi_x\in L^\infty(0,T;H_0^1(0,1)), \\
 \displaystyle  v_t\in L^\infty(0,T;L^2(0,1))\cap L^2(0,T;H^1(0,1)), \\
 \displaystyle \phi_t\in L^\infty(0,T;L^2(0,1))\cap L^2(0,T;H^1(0,1)), \\
\displaystyle  u_t,\theta_t,\phi_{xt},,u_{xx},\theta_{xx},\phi_{xxx}\in L^2((0,1)\times(0,T)).
\end{array}\right.
\end{equation}
Moreover, there exist a positive constant $C$ depending on the initial data and $T$, satisfying
\begin{equation}\label{upper and lower bound}
C^{-1}\leq v(x,t)\leq C,\ \ C^{-1}\leq \theta(x,t)\leq C,\ \ \phi(x,t)\in[-1,1].
\end{equation}
\end{theorem}

\begin{remark}\label{rem:1-1}
Our result can be regarded as a natural generalization of the Kazhikhov-Shelukhin's result (\cite{KS1977}) for the compressible single-phase flow with constant heat conductivity to the  non-isentropic compressible immiscible two-phase flow with degenerate and nonlinear heat conductivity.
\end{remark}

\begin{remark}\label{rem:1-2}
The global existence and uniqueness of strong solutions for NSAC system \eqref{NSFAC-Lagrange} means no vacuum, phase separation, shock wave, mass or heat or phase concentration will be developed in finite time, although the motion of the two-phase immiscible flow has large oscillations and the interaction between the hydrodynamic and phase-field effects is complex.
\end{remark}

Now we give some notes on the proof of the main theorem. The key to the proof is to get the upper and lower bounds of $v,\theta$ and $\phi$, see \eqref{bound of density}, \eqref{bound of phi}, \eqref{the upper bound estimate for temperature}. Inspired by the idea of Kazhikhov \cite{K1981}, we first obtain a key expression of $v$, see \eqref{expression of v}, combining with the energy inequality \eqref{basic energy inequality}, Jensen's inequality, we obtain the positive lower bound of $v$ and $\theta$. Then after getting the upper and lower bounds on $\phi$, observing that $\max_{x\in[0,1]}\big(\frac{\phi_x}{v}\big)^2(x,t)$ can be bounded by $C\big(\max_{x\in[0,1]} \theta (x,t)+1+V(t)\big)$ (see \eqref{The square modulus of the first derivative for phi}), we obtain the upper bound of $v$ by the expression of $v$. Further, after observing that the inequalities \eqref{bound of the square modulus of the first derivative} and \eqref{bound of the square modulus of the third derivative for phi}, we get the estimates on the $L^\infty(0,T;L^2)$-norm of $\theta_x$, and the upper bound of temperature $\theta$ is achieved. The details of the proof will be shown in the next section.
\section{The Proof of Theorem}
  The following Lemma is the existence and uniqueness of local strong solutions which can be obtained  by the fixed point method. From here to the end of this paper,  $C>0 $ denotes  the
generic positive constant  depending only on $\|(v_0,u_0,\theta_0\|_{H^1(0,1)}$, $\|\phi_0\|_ {H^2(0,1)}$, $\inf\limits_{x\in [0,1]}v_0(x)$, and $ \inf\limits_{x\in [0,1]}\theta_0(x)$. Moreover, by using conservation of energy, combining with \eqref{rho},  without loss of generality, the following assumption is given
\begin{equation}\label{initial energy}
  \int_0^1v_0dx=1,\ \ \int_0^1\big(\frac {u_0^2}{2}+\theta_0+\frac{1}{4\epsilon}(\phi_0^2-1)^2+\frac{\epsilon}{2}\frac{\phi_{0x}^2}{v_0}\big)dx=1.
\end{equation}
\begin{lemma}\label{Local existence}
Let \eqref{condition 1} and \eqref{condition 2}  hold. Then there exists some $T_*>0$ such that the initial boundary value problem \eqref{NSFAC-Lagrange}-\eqref{initial condition} has a unique strong solution $(v,u,\theta,\phi)$  satisfying
\begin{equation}\label{local solution space}
\left\{\begin{array}{llll}
 \displaystyle v,\theta\in L^\infty(0,T_*;H^1(0,1)),\ \ u,\phi_x\in L^\infty(0,T_*;H_0^1(0,1)), \\
 \displaystyle  v_t\in L^\infty(0,T_*;L^2(0,1))\cap L^2(0,T_*;H^1(0,1)), \\
 \displaystyle \phi_t\in L^\infty(0,T_*;L^2(0,1))\cap L^2(0,T_*;H^1(0,1)), \\
\displaystyle  u_t,\theta_t,\phi_{xt},u_{xx},\theta_{xx},\phi_{xxx}\in L^2((0,1)\times(0,T_*)).
\end{array}\right.\end{equation}
\end{lemma}

Theorem \ref{thm-global} can be achieved by extending the local solutions globally in time based on the following series of prior estimates.

\begin{lemma}\label{Fundamental energy inequality}
Let $(v,u,\theta,\phi)$ be a smooth solution of \eqref{NSFAC-Lagrange}-\eqref{initial condition} on $[0,1]\times [0,T]$. Then it holds
\begin{equation}\label{basic energy inequality}
  \sup_{0\leq t\leq T}\int_0^1\big(\frac{u^2}{2}+\frac{1}{4\epsilon}(\phi^2-1)^2+\frac{\epsilon}{2}\frac{\phi_x^2}{v}+(v-\ln v)+(\theta-\ln \theta)\big)dx+\int_0^TV(s)ds\leq E_{0},
\end{equation}
where
\begin{equation}\label{V}
  V(t)\overset{\text{def}}{=}\int_0^1\big(\frac{\theta^\beta\theta_x^2}{v\theta^2}+\frac{u_x^2}{v\theta}+\frac{v\mu^2}{\theta}\big)dx,
\end{equation}
and
\begin{equation}\label{E0}
 E_0\overset{\text{def}}{=}\int_0^1\big(\frac{u_0^2}{2}+\frac{1}{4\epsilon}(\phi_0^2-1)^2+\frac{\epsilon}{2}\frac{\phi_{0x}^2}{v}+(v_0-\ln v_0)+(\theta_0-\ln \theta_0)\big)dx.
\end{equation}
\end{lemma}
\begin{proof}  From \eqref{NSFAC-Lagrange} and \eqref{boundary condition}, combining with \eqref{initial energy}, we have
\begin{equation}\label{conservation}
  \int_0^1vdx=1,\ \ \int_0^1\big(\frac {u^2}{2}+\theta+\frac{1}{4\epsilon}(\phi^2-1)^2+\frac{\epsilon}{2}\frac{\phi_x^2}{v}\big)dx=1.
\end{equation}Multiplying \eqref{NSFAC-Lagrange}$_1$ by $1-\frac1v$, \eqref{NSFAC-Lagrange}$_2$ by $u$, \eqref{NSFAC-Lagrange}$_3$ by $\mu$, \eqref{NSFAC-Lagrange}$_5$ by $1-\frac1\theta$, and adding them together, we get
\begin{eqnarray}\label{basic inequality}
&& \big(\frac{u^2}{2}+\frac{1}{4\epsilon}(\phi^2-1)^2+\frac{\epsilon}{2}\frac{\phi_x^2}{v}+(v-\ln v)+(\theta-\ln \theta)\big)_t+\big(\frac{\theta^\beta\theta_x^2}{v\theta^2}+\frac{u_x^2}{v\theta}+\frac{v\mu^2}{\theta}\big) \notag\\
&&=u_x+\big(\frac{uu_x}{v}-\frac{u\theta}{v}\big)_x+\big((1-\theta^{-1})\frac{\theta^\beta\theta_x}{v}\big)_x+\epsilon\big(\frac{\phi_x\phi_t}{v}\big)_x-\frac{\epsilon}{2}\big(\frac{\phi_x^2u}{v^2}\big)_x,
\end{eqnarray}
integrating \eqref{basic inequality} over $[0,1]\times[0,T]$ by parts, together with the boundary condition \eqref{boundary condition}, we obtain \eqref{basic energy inequality}, and the proof of Lemma \ref{Fundamental energy inequality} is finished.
\end{proof}

\begin{lemma}\label{lem-expression of v}
Let $(v,u,\theta,\phi)$ be a smooth solution of \eqref{NSFAC-Lagrange}-\eqref{initial condition} on $[0,1]\times [0,T]$. Then it has the following expression of $v$
\begin{equation}\label{expression of v}
  v(x,t)=   D(x,t) Y(t) + \int_0^t\frac{D(x,t) Y(t)\big(\theta(x,\tau)+\frac{\epsilon}{2}\frac{\phi^2_x(x,\tau)}{v(x,\tau)}\big)}{ D(x,\tau) Y(\tau)}d\tau,
\end{equation}
 where
 \begin{equation}\label{D}
    D(x,t)=v_0(x)e^{\int_{0}^x \big(u(y,t)-u_0(y)\big)dy} e^{\big(-\int_0^1v\int_0^xudydx+\int_0^1v_0\int_0^xu_0dydx\big)},
 \end{equation}
 and
 \begin{equation}\label{Y}
  Y(t)=e^{-\int_0^t\int_0^1\big(u^2+\theta+\frac{\epsilon}{2}\frac{\phi^2_x}{v}\big) dxds}.
 \end{equation}
\end{lemma}
\begin{proof} We  introduce the following  $\sigma$
  \begin{equation}\label{sigma}
  \sigma\overset{\text{def}}{=}\frac{u_x}{v}-\frac{\theta}{v}-\frac{\epsilon}{2}\frac{\phi_x^2}{v^2}.
 \end{equation}
Firstly, integrating \eqref{NSFAC-Lagrange}$_2$ over $[0,x]$, we have
\begin{equation}\label{Momentum equation form 2}
 \left(\int_0^xudy\right)_t=\sigma-\sigma(0,t),
\end{equation}
multiply both sides of   \eqref{Momentum equation form 2}  by $v$, we get
\begin{equation*}
  v\sigma(0,t)=v\sigma-v\left(\int_0^xudy\right)_t,
\end{equation*}
Integrating the above equation over $[0,1]$, combining with \eqref{initial energy} and \eqref{boundary condition}, we obtain
\begin{eqnarray}\label{sigma 0}
 \sigma(0,t) &=&\int_0^1v\sigma dx-\int_0^1v\left(\int_0^xudy\right)_tdx \notag\\
   &=&\int_0^1\big(u_x-\theta-\frac{\epsilon}{2}\frac{\phi_x^2}{v}\big)dx-\big(\int_0^1v\int_0^xudydx\big)_t+\int_0^1u_x\int_0^xudy dx \notag\\
   &=& - \left(\int_0^1v\int_0^xudydx\right)_t-\int_0^1\left(u^2+\theta+\frac{\epsilon}{2}\frac{\phi_x^2}{v} \right) dx.
\end{eqnarray}
Secondly, from \eqref{NSFAC-Lagrange}$_1$, we have
\begin{equation*}
  \sigma=(\ln v)_t-\frac{\theta}{v}-\frac{\epsilon}{2}\frac{\phi^2_x}{v^2},
\end{equation*}
which combining with \eqref{Momentum equation form 2} and \eqref{sigma 0}, we get
\begin{equation*}
  \left(\int_0^xudy\right)_t  = (\ln v)_t-\frac{\theta}{v}-\frac{\epsilon}{2}\frac{\phi^2_x}{v^2}+ \left(\int_0^1v\int_0^xudydx\right)_t+\int_0^1\left(\theta+ u^2 + \frac{\epsilon}{2}\frac{\phi^2_x}{v}\right) dx,
\end{equation*}
integrating both sides of the above equation with respect to $t$, we have
\begin{equation}\label{v}
 v(x,t) = D(x,t) Y(t) e^{\int_0^t\big(\frac{\theta}{v}+\frac{\epsilon}{2}\frac{\phi^2_x}{v^2}\big)ds},
\end{equation}
with $D(x,t)$ and $Y(t)$ as defined in  \eqref{D} and \eqref{Y}  respectively.

Finally, introducing the following  $g$
\begin{equation}\label{g}
  g =\int_0^t\big(\frac{\theta}{v}+\frac{\epsilon}{2}\frac{\phi^2_x}{v^2}\big)ds,
\end{equation}
by using \eqref{v}, we  get the following ordinary differential equation for $g$
\begin{equation*}
  g_t=\frac{\theta(x,t)+\frac{\epsilon}{2}\frac{\phi^2_x(x,t)}{v(x,t)}}{v(x,t)}=\frac{\theta(x,t)+\frac{\epsilon}{2}\frac{\phi^2_x(x,t)}{v(x,t)}}{ D(x,t) Y(t)e^g},
\end{equation*}
and this gives us an expression for $e^g$
\begin{equation*}
  e^g  =1+\int_0^t\frac{\theta(x,\tau)+\frac{\epsilon}{2}\frac{\phi^2_x(x,\tau)}{v(x,\tau)}}{ D(x,\tau) Y(\tau)}d\tau,
\end{equation*}
substituting it into the expression \eqref{v}, and thus the proof of
Lemma \ref{lem-expression of v} is finished.
\end{proof}

\begin{lemma}\label{the lower bounds of density and temperature}
Let $(v,u,\theta,\phi)$ be a smooth solution of \eqref{NSFAC-Lagrange}-\eqref{initial condition} on $[0,1]\times [0,T]$. Then it holds that for $\forall(x,t)\in[0,1]\times [0,T]$
 \begin{equation}\label{bound of density}
 v(x,t)\geq C^{-1},\qquad \theta(x,t)\geq C^{-1}.
 \end{equation}
\end{lemma}
\begin{proof}
Firstly, since the function $x-\ln x $ is convex, by using Jensen's inequality, we have
\begin{equation}\label{Jensen}
  \int_0^1\theta dx-\ln\int_{0}^{1}   \theta dx  \le \int_{0}^{1} (\theta-\ln\theta ) dx,
\end{equation}
combining with \eqref{basic energy inequality} and \eqref{conservation}, we get
 \begin{equation}\label{bar theta}
 \bar \theta(t)\overset{\text{def}}{=} \int_0^{1}\theta(x,t)dx\in[ \alpha_1,1],
 \end{equation}
 where $0<\alpha_1<\alpha_2$  are  the two roots of the following algebraic equation
 \begin{equation}
   x-\ln x =E_0.
 \end{equation}

Secondly, from \eqref{conservation}, by using Cauchy's inequality, we have
 \begin{equation*}
   \big|\int_0^1 v\int_0^x udydx\big|\le \int_0^1 v\big|\int_0^x udy\big|dx\le\int_0^1 v\big(\int_0^1 u^2dy\big)^{1/2}dx\le C,
 \end{equation*}
 and then by the definition \eqref{D} of $D$, we get
\begin{equation}\label{upper and lower bound for D}
 C^{-1}\leq D(x,t)\leq C.
\end{equation}
Moreover,  from \eqref{conservation}, combining with \eqref{bar theta}, we also have
\begin{equation*}
\big|\int_0^1\ln vdx\big|+ \int_0^1\big(u^2+\theta+\frac{\epsilon}{2}\frac{\phi_x^2}{v}\big)dx\leq C,
\end{equation*}
which follows that, for $\forall \tau\in[0,t)$,
\begin{equation}\label{upper and lower bound for Y}
  e^{-2t}\leq Y(t)\leq1,\ \ \mathrm{and}\ \ e^{-2(t-\tau)}\leq\frac{Y(t)}{Y(\tau)}\leq e^{-\alpha_1(t-\tau)}.
\end{equation}
By using \eqref{v}, \eqref{upper and lower bound for D} and \eqref{upper and lower bound for Y}, we obtain, there exist some positive constant $C$, such that
  \begin{equation}\label{the lower bound of v}
   v(x,t)\geq C^{-1},\qquad \forall(x,t)\in[0,1]\times[0,T].
  \end{equation}

Finally, $\forall p>2$, multiplying \eqref{NSFAC-Lagrange}$_5$ by $\theta^{-p}$, integrate over $[0,1]$  with respect to $x$, by using \eqref{the lower bound of v},  we have
\begin{align*}
\frac{1}{p-1}\frac{d}{dt}\int_0^1\left( {\theta}^{-1}\right)^{p-1}dx+\int_0^1\frac{u_x^2}{v\theta^p}dx
 &\leq\int_0^1\frac{u_x}{v\theta^{p-1}}dx\notag\\
 & \leq\frac{1}{2}\int_0^1\frac{u_x^2}{v\theta^p}dx+\frac{1}{2}\int_0^1\frac{1}{v\theta^{p-2}}dx\notag\\
 &\leq\frac{1}{2}\int_0^1\frac{ u_x^2}{v\theta^p}dx+C\left\|  \theta^{-1} \right\|^{p-2}_{L^{p-1}}.
\end{align*}
Applying Gronwall's inequality to  the above result, we obtain
\begin{equation*}
  \sup_{0\leq t\leq T}\left\|  \theta^{-1}(\cdot,t) \right\|_{L^{p-1}}\leq C,\ \ \forall p>2,
\end{equation*}
moreover,  letting $p$ tends to infinity, we do eventually get the  lower bound of $\theta$. The proof of Lemma \ref{the lower bounds of density and temperature} is completed.
  \end{proof}

\begin{lemma}\label{the upper bounds of density}
Let $(v,u,\theta,\phi)$ be a smooth solution of \eqref{NSFAC-Lagrange}-\eqref{initial condition} on $[0,1]\times [0,T]$. Then it holds that for $\forall(x,t)\in[0,1]\times [0,T]$
 \begin{equation}\label{bound of phi}
|\phi(x,t)|\leq C,\qquad 
v(x,t)\leq C.
\end{equation}
\end{lemma}
\begin{proof}
Firstly,  from \eqref{basic energy inequality}, we have
 \begin{equation*}
   \frac{\epsilon}{4}\int_0^1(\phi^2-1)^2\textcolor[rgb]{1.00,0.00,0.00}{dx}\leq  E_0,
 \end{equation*}
 which implies that
  \begin{equation}\label{phi L4}
   \int_0^1\phi^4\textcolor[rgb]{1.00,0.00,0.00}{dx}\leq C.
 \end{equation}
Moreover, for $\forall (x,t)\in[0,1]\times[0,T]$,
\begin{eqnarray}\label{lower bound of phi}
  |\phi(x,t)|&\leq&\left|\int_0^1\big(\phi(x,t)-\phi(y,t)\big)dy\right|+\left|\int_0^1\phi(y,t)dy\right|\notag \\
 &\leq&\left|\int_0^1\big(\int_y^x\phi_\xi(\xi,t)d\xi\big)dy\right| +CE_0\notag\\
  &\leq& \big(\int_0^1\frac{\phi_x^2}{v}dx\big)^{\frac12}+CE_0\notag\\
  &\leq& C.
\end{eqnarray}

Next, for $0<\alpha< 1 $ and $0<\varepsilon<1,$  integrating \eqref{NSFAC-Lagrange}$_5$ multiplied by $\theta^{-\alpha}$ over $ (0,1)\times(0,T)$ yields
\begin{align}\label{theta11}
&\int_0^T\int_0^1\frac{\alpha\theta^\beta\theta_x^2}{v\theta^{\alpha+1}}dxdt+\int_0^T\int_0^1\frac{u_x^2 +(v\mu)^2}{v\theta^\alpha}dxdt \nonumber\\
& =  \frac{1}{1-\alpha}\int_0^1\left(\theta^{1-\alpha}-\theta_0^{1-\alpha}\right)dx+\int_0^T\int_0^1\frac{\theta^{1-\alpha}u_x}{v} dxdt \nonumber\\
&  \leq C(\alpha)+\frac{1}{2}\int_0^T\int_0^1\frac{u_x^2}{v\theta^\alpha}dxdt+C\int_0^T\int_0^1 \theta^{2-\alpha} \nonumber dxdt \\
&\leq C(\alpha)+\frac{1}{2}\int_0^T\int_0^1\frac{u_x^2}{v\theta^\alpha}dxdt
+C\int_0^T\max_{x\in[0,1]}\theta^{1-\alpha}\int_0^1\theta dxdt \nonumber\\
&\leq C(\alpha,\varepsilon)+\frac{1}{2}\int_0^T\int_0^1\frac{u_x^2}{v\theta^\alpha}dxdt+\varepsilon \int_0^T\max_{x\in[0,1]}\theta dt,
\end{align}
where in the first inequality we have used \eqref{basic energy inequality} and \eqref{bound of density}.
Then,  for $ \alpha=\min\{1,\beta\}/2,$ using \eqref{bound of density}, we get
\begin{align}\label{theta12}
\int_0^T\max_{x\in[0,1]}\theta dt &\leq C+C\int_0^T\int_0^1|\theta_x|dxdt \nonumber\\
&\leq C+C\int_0^T\int_0^1\frac{ \theta^\beta\theta_x^2}{v\theta^{1+\alpha}}dxdt+C\int_0^T\int_0^1\frac{v\theta^{1+\alpha}}{\theta^\beta}dxdt \nonumber\\
&\leq C +C\int_0^T\int_0^1\frac{ \theta^\beta\theta_x^2}{v\theta^{1+\alpha}}dxdt+\frac{1}{2}\int_0^T\max_{x\in[0,1]}\theta dt,
\end{align}
which together with \eqref{theta11} yields that
\begin{align}\label{theta13}
\displaystyle  \int_0^T\max_{x\in[0,1]}\theta dt\leq C,
\end{align}
and then that for $0<\alpha< 1, $
\begin{align}\label{theta14}
\displaystyle  \int_0^T\int_0^1\frac{ \theta^\beta\theta_x^2}{v\theta^{\alpha+1}}dxdt \leq C(\alpha).
\end{align}


Finally, we can give the upper bounds  of $v$. In fact,  combining the  expression of $v$ \eqref{expression of v}  with the upper and lower bound estimates \eqref{upper and lower bound for D}--\eqref{the lower bound of v}, we have the following inequality
\begin{align}\label{the upper bound estimate for v 1}
  v(x,t)&=D(x,t) Y(t) + \int_0^t\frac{D(x,t) Y(t)\big(\theta(x,\tau)+\frac{\epsilon}{2}\frac{\phi^2_x(x,\tau)}{v(x,\tau)}\big)}{ D(x,\tau) Y(\tau)}d\tau \notag\\
&\le C+C\int_0^t e^{-\alpha_1(t-\tau)}\Big(\max_{x\in[0,1]} \theta (x,\tau)+\max_{x\in[0,1]}\big(\frac{\phi_x(x,\tau)}{v(x,\tau)}\big)^2\max_{x\in[0,1]}v(x,\tau)\Big)d\tau.
\end{align}
In view of \eqref{NSFAC-Lagrange}$_4$,  we derive that
\begin{equation}\label{second derivative form for phi}
 \epsilon\Big(\frac{\phi_x}{v}\Big)_x=-\mu+\frac{1}{\epsilon}(\phi^3-\phi),
\end{equation}
and then combining with \eqref{basic energy inequality}, \eqref{conservation}, \eqref{bound of density}, \eqref{lower bound of phi}, we obtain
\begin{equation}\label{estimate for second derivative form of phi}
\int_0^1\Big(\frac{\phi_x}{v}\Big)_x^2\frac{v}{\theta}dx\leq C\big(1+V(t)\big).
\end{equation}
Using \eqref{basic energy inequality}, \eqref{conservation}, \eqref{the lower bound of v}, \eqref{lower bound of phi}, \eqref{estimate for second derivative form of phi}, we get
\begin{eqnarray}\label{The square modulus of the first derivative for phi}
 \max_{x\in[0,1]}\big(\frac{\phi_x}{v}\big)^2(x,t)& \leq& C\int_0^1\frac{\phi_x}{v}\Big(\frac{\phi_x}{v}\Big)_xdx\notag\\
 &\leq&C\int_0^1\frac{\theta}{v^2}\frac{\phi_x^2}{v}dx+\int_0^1\Big(\frac{\phi_x}{v}\Big)_x^2\frac{v}{\theta}dx\notag\\
  &\leq& C\big(\max_{x\in[0,1]} \theta (x,t)+1+V(t)\big).
\end{eqnarray}
 Substituting \eqref{The square modulus of the first derivative for phi} into \eqref{the upper bound estimate for v 1}, we obtain
 \begin{eqnarray}\label{the upper bound estimate for v}
  v(x,t)\leq C+C\int_0^t\big(\max_{x\in[0,1]} \theta (x,t)+1+V(t)\big)\max_{x\in[0,1]} v(x,\tau)d\tau,
\end{eqnarray}
by using Gronwall's inequality, we get the upper bound of $v$
\begin{equation}\label{the upper bound of v}
 v(x,t)\leq C,\qquad\forall (x,t)\in [0,1]\times [0,T].
\end{equation}
The proof of Lemma \ref{the upper bounds of density} is completed.
\end{proof}

\begin{lemma}\label{the square modulus estimate for density}
Let $(v,u,\theta,\phi)$ be a smooth solution of \eqref{NSFAC-Lagrange}-\eqref{initial condition} on $[0,1]\times [0,T]$. Then it holds that for $\forall(x,t)\in[0,1]\times [0,T]$
 \begin{equation}\label{bound of the square modulus of the first derivative}
\sup_{0\leq t\leq T}\int_0^1v_x^2dx\leq C,\qquad \int_0^T\int_0^1\phi_{xx}^2dx\leq C,\qquad\int_0^T\int_0^1\phi_t^2dx\leq C.
 \end{equation}
\end{lemma}
\begin{proof} Firstly, we rewrite the momentum equation  in Lagrange coordinate system \eqref{NSFAC-Lagrange}$_2$ as follows
\begin{equation}\label{Another form of the momentum equation}
 \Big(u-\frac{v_x}{v}\Big)_t=-\Big(\frac{\theta}{v}+\frac\epsilon2\big(\frac{\phi_x}{v}\big)^2\Big)_x
\end{equation}
Multiplying \eqref{Another form of the momentum equation} by $u-\frac{v_x}{v}$, integrating by parts over $[0,1]$  with respect to $x$, we have
\begin{eqnarray}\label{the square modulus estimate for density-1}
&&\frac12\int_0^1\Big(u-\frac{v_x}{v}\Big)^2(x,t)dx-\frac12\int_0^1\Big(u-\frac{v_x}{v}(x,0)\Big)^2dx \notag\\
&&=\int_0^t\int_0^1\Big(\frac{\theta v_x}{v^2}-\frac{\theta_x}{v}-\epsilon\frac{\phi_x}{v}\big(\frac{\phi_x}{v}\big)_x\Big)\Big(u-\frac{v_x}{v}\Big)dxdt\notag \\
&&=-\int_0^t\int_0^1\frac{\theta v_x^2}{v^3}dxdt+\int_0^t\int_0^1\frac{\theta u v_x}{v^2}dxdt \notag\\
&&\ \ \ \ \ -\int_0^t\int_0^1\frac{\theta_x}{v}\Big(u-\frac{v_x}{v}\Big)dxdt
-\int_0^t\int_0^1\epsilon\frac{\phi_x}{v}\big(\frac{\phi_x}{v}\big)_x\Big(u-\frac{v_x}{v}\Big)dxdt.
 \end{eqnarray}
Now we give  the last three terms on the right side of \eqref{the square modulus estimate for density-1}. First, by using \eqref{basic energy inequality}, \eqref{the lower bound of v}, \eqref{theta13} and \eqref{the upper bound of v},  we have
\begin{align}\label{I1}
 \Big|\int_0^t\int_0^1\frac{\theta u v_x}{v^2}dxdt\Big|& \leq\frac{1}{2}\int_0^t\int_0^1\frac{\theta v_x^2}{v^3}dxdt+\frac{1}{2}\int_0^t\int_0^1\frac{u^2\theta}{v}dxdt\notag\\
& \leq\frac{1}{2}\int_0^t\int_0^1\frac{\theta v_x^2}{v^3}dxdt+C\int_0^t\max_{x\in[0,1]}\theta dt\notag\\
& \leq\frac{1}{2}\int_0^t\int_0^1\frac{\theta v_x^2}{v^3}dxdt+C.
\end{align}
Next, for the third term on the righthand side of \eqref{the square modulus estimate for density-1}, by using \eqref{basic energy inequality} and \eqref{bound of density} we obtain
\begin{align}\label{I2}
\Big| \int_0^t\int_0^1\frac{\theta_x}{v}\Big(u-\frac{v_x}{v}\Big)dxdt\Big|
&\leq \int_0^t\int_0^1\frac{\theta^\beta\theta_x^2}{v\theta^2}dxdt+\frac{1}{2} \int_0^t\int_0^1\frac{\theta^2}{v\theta^\beta}\Big(u-\frac{v_x}{v}\Big)^2dxdt\notag\\
&\leq C+C\int_0^t\max_{x\in[0,1]}\theta^2\int_0^1\Big(u-\frac{v_x}{v}\Big)^2dxdt.
\end{align}
Moreover, it follows from \eqref{basic energy inequality}, \eqref{bound of density}, \eqref{bound of phi},  \eqref{theta13} and \eqref{theta14} that, for any $\varepsilon>0,$
\begin{align}\label{theta 2}
  \int_0^T\max_{x\in[0,1]}\theta^2dt
  &\leq C\int_0^T \max_{x\in[0,1]}\Big|\theta^2-\int_0^1\theta^2dx\Big|dt+C\int_0^T\max_{x\in[0,1]}\theta dt\notag \\
  &\leq C+\int_0^T\int_0^1\theta|\theta_x|dxdt \notag\\
  &\leq C+C\int_0^T\int_0^1\frac{\theta^\beta\theta_x^2}{v\theta^{\alpha+1}}dxdt+\varepsilon\int_0^T\int_0^1v\theta^{3+\alpha-\beta}dxdt \notag\\
   &\leq C+C\varepsilon\int_0^1\max_{x\in[0,1]}\theta^2dxdt,
 \end{align}
which follows that
\begin{equation}\label{max theta^2}
 \int_0^T\max_{x\in[0,1]}\theta^2dt \leq C.
\end{equation}
Moreover, integrating \eqref{NSFAC-Lagrange}$_5$ over $[0,1]\times[0,T]$, combining with \eqref{the lower bound of v}, \eqref{the upper bound of v}, \eqref{max theta^2}, we have
\begin{eqnarray}
  \int_0^T\int_0^1\frac{u_x^2+(v\mu)^2}{v}dxdt&=&\int_0^1\theta dx-\int_0^1\theta_0dx+\int_0^T\int_0^1\frac{\theta}{v}u_x dx\notag\\
   &\leq& C+\frac{1}{2}\int_0^T\int_0^1\frac{u_x^2}{v}dxdt,
\end{eqnarray}
which implies that
\begin{equation}\label{First derivative squared modulus}
 \int_0^T\int_0^1\big(u_x^2+(v\mu)^2\big)dxdt\leq C,
\end{equation}
Finally, for the fourth term on the righthand side of \eqref{the square modulus estimate for density-1}, by using \eqref{bound of density}, \eqref{estimate for second derivative form of phi}, \eqref{the upper bound of v}, we get
\begin{eqnarray}\label{I3}
 &&\Big|\int_0^t\int_0^1\epsilon\frac{\phi_x}{v}\big(\frac{\phi_x}{v}\big)_x\Big(u-\frac{v_x}{v}\Big)dxdt\Big|\notag\\
 &&\leq C\int_0^t\int_0^1\Big(\big|(\frac{\phi_x}{v})_x\big|^2+\big|\frac{\phi_x}{v}\big|^2\big(u-\frac{v_x}{v}\big)^2\Big)dxdt\notag\\
 && \leq C+C\int_0^t\max_{x\in[0,1]}\big|\frac{\phi_x}{v}\big|^2\int_0^1\big(u-\frac{v_x}{v}\big)^2dxdt
\end{eqnarray}
Substituting \eqref{I1}, \eqref{I2}, \eqref{I3} into \eqref{the square modulus estimate for density-1}, combining with \eqref{The square modulus of the first derivative for phi}, \eqref{max theta^2} and Gronwall's inequality, we have
\begin{equation}\label{estimate for v-x^2}
\sup_{0\leq t\leq T} \int_0^1\Big(u-\frac{v_x}{v}\Big)^2dx+\int_0^T\int_0^1\frac{\theta v_x^2}{v^3}dxdt\leq C,
\end{equation}
together with \eqref{basic energy inequality}, \eqref{bound of density} and \eqref{bound of phi},  we derive that
 \begin{equation*}
   \sup_{0\leq t\leq T}\int_0^1v_x^2dx\leq C.
 \end{equation*}

Secondly, we rewrite  \eqref{NSFAC-Lagrange}$_{3,4}$ as follows
\begin{equation}\label{Allen-Cahn-Lagrange}
 \phi_t-\epsilon\phi_{xx}=-\epsilon\frac{\phi_xv_x}{v}-\frac{v}{\epsilon}\big(\phi^3-\phi\big),
\end{equation}
multiplying \eqref{Allen-Cahn-Lagrange} by $\phi_{xx}$ and integrating the resultant over $[0,1]$, by using \eqref{bound of density}, \eqref{bound of phi}, \eqref{The square modulus of the first derivative for phi},  we obtain
\begin{align}
&\quad\frac{1}{2}\frac{d}{dt}\int_0^1\phi_x^2dx+\epsilon\int_0^1\phi_{xx}^2dx\notag\\
&=\epsilon\int_0^1\frac{\phi_xv_x}{v}\phi_{xx}dx+ \frac{1}{\epsilon}\int_0^1 v(\phi^3-\phi)\phi_{xx}dx\notag\\
& \leq C\Big(\int_0^1\phi_x^2v_x^2dx+\int_0^1 (\phi^3-\phi)^2dx\Big)+\frac{\epsilon}{2}\int_0^1\phi_{xx}^2dx\notag\\
&\leq  C\Big(\max_{x\in[0,1]}\phi_x^2(x,t)\int_0^1v_{x}^2dx+\int_0^1\phi^2dx\Big)+\frac{\epsilon}{2}\int_0^1\phi_{xx}^2dx\notag\\
&\leq   C\big(\max_{x\in[0,1]} \theta (x,t)+1+V(t)\big)+\frac{\epsilon}{2}\int_0^1\phi_{xx}^2dx,
\end{align}
combining with \eqref{basic energy inequality} and \eqref{theta13}, we have
\begin{equation}\label{phi-xx-L2}
  \int_0^T\int_0^1\phi_{xx}^2dx\leq C.
\end{equation}

Finally, for $\phi_t$, from \eqref{Allen-Cahn-Lagrange}, we get
\begin{equation}\label{Allen-Cahn-Lagrange-0}
 \phi_t=\epsilon\phi_{xx}-\epsilon\frac{\phi_xv_x}{v}-\frac{v}{\epsilon}\big(\phi^3-\phi\big),
\end{equation}
integrating \eqref{Allen-Cahn-Lagrange-0} over $[0,1]$, we arrive
\begin{eqnarray}\label{estimate of phi-xx-l2}
 \int_0^1\phi_t^2dx&\leq&C\Big(\int_0^1\phi_{xx}^2dx+\int_0^1\phi_x^2v_x^2dx+\int_0^1\big(\phi^3-\phi\big)^2dx\Big)\notag\\
 &\leq&C\Big(\int_0^1\phi_{xx}^2dx+\int_0^1v_x^2dx\int_0^1\phi_{xx}^2dx+1\Big)\notag\\
 &\leq&C\Big(\int_0^1\phi_{xx}^2dx+1\Big),
\end{eqnarray}
then, by using \eqref{phi-xx-L2}, we arrive
\begin{equation}\label{phi-t-L2}
  \int_0^T\int_0^1\phi_t^2dx\leq C.
\end{equation}
The proof of Lemma \ref{the square modulus estimate for density} is finished.
\end{proof}

In order to get the a priori estimates on $u$, we present some prior estimates for higher derivatives of $\phi$.
\begin{lemma}\label{phi-xxx}
Let $(v,u,\theta,\phi)$ be a smooth solution of \eqref{NSFAC-Lagrange}-\eqref{initial condition} on $[0,1]\times [0,T]$. Then it holds that for $\forall(x,t)\in[0,1]\times [0,T]$
 \begin{equation}\label{bound of the square modulus of the third derivative for phi}
\sup_{0\leq t\leq T}\int_0^1\phi_{xx}^2dx+\int_0^T\int_0^1\big(\phi_{xt}^2+\big(\frac{\phi_{x}}{v}\big)_{xx}^2\big)dx\leq C.
 \end{equation}
\end{lemma}
\begin{proof}
We rewrite  \eqref{Allen-Cahn-Lagrange} as
\begin{equation}\label{Allen-Cahn-Lagrange-1}
\frac{\phi_t}{v}-\epsilon\Big(\frac{\phi_x}{v}\Big)_{x}=-\frac{1}{\epsilon}(\phi^3-\phi),
\end{equation}
then, differentiating \eqref{Allen-Cahn-Lagrange-1} with respect to $x$, we have
\begin{equation}\label{Allen-Cahn-Lagrange-2}
 \Big(\frac{\phi_x}{v}\Big)_t-\epsilon\Big(\frac{\phi_x}{v}\Big)_{xx}=-\frac{1}{\epsilon}\big(\phi^3-\phi\big)_x+\frac{\phi_tv_x}{v^2}-\frac{\phi_x u_x}{v^2},
\end{equation}
multiplying \eqref{Allen-Cahn-Lagrange-2} by $\big(\frac{\phi_{x}}{v}\big)_t$ and integrating the resultant over $[0,1]$, by using \eqref{basic energy inequality}, \eqref{bound of density}, \eqref{bound of phi}, \eqref{bound of the square modulus of the first derivative}, \eqref{First derivative squared modulus} and \eqref{estimate of phi-xx-l2}, we obtain
\begin{eqnarray}
&&\int_0^1\Big(\frac{\phi_{x}}{v}\Big)_t^2dx+ \frac{\epsilon}{2}\frac{d}{dt}\int_0^1\Big(\frac{\phi_x}{v}\Big)_{x}^2dx \notag\\
&& =-\frac{1}{\epsilon}\int_0^1\big(\phi^3-\phi\big)_x\big(\frac{\phi_{x}}{v}\big)_tdx+\int_0^1\frac{\phi_tv_x}{v^2}\big(\frac{\phi_{x}}{v}\big)_tdx-\int_0^1\frac{\phi_x u_x}{v^2}\big(\frac{\phi_{x}}{v}\big)_tdx\notag\\
&&\leq C\Big(\int_0^1\big(3\phi^2-1\big)^2\phi_x^2dx+\int_0^1\phi_t^2v_x^2dx+\int_0^1\phi_x^2 u_x^2dx\Big)+\frac{1}{3}\int_0^1\big(\frac{\phi_{x}}{v}\big)_t^2dx \notag\\
&&\leq C\Big(1+\|\phi_t\|_{L^\infty}^2\int_0^1v_x^2dx+\|\frac{\phi_{x}}{v}\|_{L^\infty}^2\int_0^1u_x^2dx\Big)
+\frac{1}{3}\int_0^1\big(\frac{\phi_{x}}{v}\big)_t^2dx\notag\\
&&\leq C\Big(1+\int_0^1\big(\phi_t^2+2|\phi_t\phi_{xt}|\big)dx+\int_0^1\Big(\frac{\phi_x}{v}\Big)_{x}^2dx\int_0^1u_x^2dx\Big)
+\frac{1}{3}\int_0^1\big(\frac{\phi_{x}}{v}\big)_t^2dx\notag\\
&&\leq C\Big(1+\int_0^1\phi_{xx}^2dx+\int_0^1\phi_{xt}^2dx+\int_0^1\Big(\frac{\phi_x}{v}\Big)_{x}^2dx\int_0^1u_x^2dx\Big)
+\frac{1}{3}\int_0^1\big(\frac{\phi_{x}}{v}\big)_t^2dx\notag\\
&&\leq C\Big(1+\int_0^1\phi_{xx}^2dx+\int_0^1u_x^2dx\int_0^1\Big(\frac{\phi_x}{v}\Big)_{x}^2dx\Big)
+\frac{1}{2}\int_0^1\big(\frac{\phi_{x}}{v}\big)_t^2dx,
\end{eqnarray}
where in the last inequality we have used $\phi_{xt}=\big(\big(\frac{\phi_x}{v}\big)_t+\frac{\phi_x u_x}{v^2}\big)v$. Thus, by Gronwall's inequality, we have
\begin{equation}\label{higher derivative energy estimation for phi}
 \sup_{t\in[0,T]} \int_0^1\Big(\frac{\phi_x}{v}\Big)_{x}^2dx+\int_0^T\int_0^1\Big(\frac{\phi_{x}}{v}\Big)_t^2dx\leq C,
\end{equation}
combining with  \eqref{basic energy inequality}, \eqref{bound of the square modulus of the first derivative}, we  get
\begin{equation}
\sup_{0\leq t\leq T}\int_0^1\phi_{xx}^2dx+\int_0^T\int_0^1\phi_{xt}^2dx\leq C.
 \end{equation}
Moreover,  by using the Sobolev embedding theorem for one dimension, \eqref{higher derivative energy estimation for phi} also implies
\begin{equation}\label{the upper and lower bounds of the derivative for phi}
  \max_{(x,t)\in[0,1]\times[0,T]}\big|\frac{\phi_x}{v}\big|^2\leq C+C\sup_{t\in[0,T]} \int_0^1\Big(\frac{\phi_x}{v}\Big)_{x}^2dx\leq C.
\end{equation}
Further, by using \eqref{Allen-Cahn-Lagrange-2} and the estimates obtained above, we achieve
\begin{equation}
\int_0^T\int_0^1\big(\frac{\phi_{x}}{v}\big)_{xx}^2dx\leq C,
 \end{equation}
  the proof of Lemma \ref{phi-xxx} is completed.
\end{proof}

After the above preparation work, now we can give a prior estimate of the velocity $u$.
\begin{lemma}\label{velocity-x}
Let $(v,u,\theta,\phi)$ be a smooth solution of \eqref{NSFAC-Lagrange}-\eqref{initial condition} on $[0,1]\times [0,T]$. Then it holds that for $\forall(x,t)\in[0,1]\times [0,T]$
 \begin{equation}\label{energy estimate of velocity}
\sup_{0\leq t\leq T}\int_0^1u_{x}^2dx+\int_0^T\int_0^1\big(u_{t}^2+u^2_{xx}\big)dx\leq C.
 \end{equation}
\end{lemma}
\begin{proof}
Multiplying \eqref{NSFAC-Lagrange}$_2$ by $u_{xx}$ and integrating the resultant over $(0,1)\times(0,T)$, by using \eqref{bound of density}, \eqref{bound of phi}, \eqref{bound of the square modulus of the first derivative}, \eqref{max theta^2}, \eqref{higher derivative energy estimation for phi}, \eqref{the upper and lower bounds of the derivative for phi},  we obtain
\begin{eqnarray}\label{u-x and u-xx 1}
 && \frac{1}{2}\int_0^1u_x^2dx+\int_0^T\int_0^1\frac{u_{xx}^2}{v}dxdt\notag  \\
 &&\leq C+\frac{1}{2}\int_0^T\int_0^1\frac{u_{xx}^2}{v}dxdt+C\int_0^T\int_0^1\Big(\theta_x^2+\theta^2 v_x^2+\big|\frac{\phi_x}{v}\big|^2\big|\big(\frac{\phi_x}{v}\big)_x\big|^2+u^2_xv_x^2\Big)dxdt\notag\\
 &&\leq C+\frac{1}{2}\int_0^T\int_0^1\frac{u_{xx}^2}{v}dxdt+C\int_0^T\int_0^1\theta_x^2dx+C\int_0^T\max_{x\in[0,1]}\theta^2\int_0^1v_x^2dxdt\notag\\
 &&\ \ +C\max_{(x,t)\in[0,1]\times[0,T]}\big|\frac{\phi_x}{v}\big|^2\int_0^T\int_0^1\big|\big(\frac{\phi_x}{v}\big)_x\big|^2dxdt+C\int_0^T\max_{x\in[0,1]}u_x^2\int_0^1v_x^2dxdt\notag\\
 &&\leq C+\frac{3}{4}\int_0^T\int_0^1\frac{u_{xx}^2}{v}dxdt+C_1\int_0^T\int_0^1\frac{\theta^\beta\theta_x^2}{v}dxdt,
\end{eqnarray}
where we have used the following and the Sobolev embedding inequality in one dimension
\begin{eqnarray}\label{u-x L2}
  \int_0^T\max_{x\in[0,1]} u_x^2dt&\leq& C(\delta)\int_0^T\int_0^1u_x^2dxdt+\delta\int_0^T\int_0^1\frac{u_{xx}^2}{v}dxdt\notag\\
  &\leq& C(\delta)+\delta\int_0^T\int_0^1\frac{u_{xx}^2}{v}dxdt.
\end{eqnarray}
Multiplying \eqref{NSFAC-Lagrange}$_5$ by $\theta$, integrating the resultant over $(0,1)\times(0,T)$, by using \eqref{max theta^2}, \eqref{First derivative squared modulus} and \eqref{u-x L2}, we have
\begin{eqnarray}\label{theta theta-x 1}
 && \frac{1}{2}\int_0^1\theta^2dx+\int_0^T\int_0^1\frac{\theta^\beta\theta_x^2}{v}dxdt \notag \\
 && \leq C+C\int_0^T\int_0^1\theta^2|u_x|dx+C\int_0^T\int_0^1\big(u_x^2+\mu^2\big)\theta dxdt\notag\\
 &&\leq C+C\int_0^T\int_0^1\theta u_x^2dxdt++C\int_0^T\int_0^1\theta^3dxdt+\int_0^T\max_{x\in[0,1]}\theta dt\notag\\
 &&\leq C+C\int_0^T\max_{x\in[0,1]} u_x^2dt+C\int_0^T\max_{x\in[0,1]} \theta^2dt\notag\\
 && \leq C(\delta)+C\delta\int_0^T\int_0^1\frac{u_{xx}^2}{v}dxdt.
\end{eqnarray}
From \eqref{u-x and u-xx 1} and \eqref{theta theta-x 1}, choosing $\delta$ small enough, we obtain
\begin{equation}\label{tempeture and velocity-1}
  \sup_{t\in[0,T]}\int_0^1\big(\theta^2+u_x^2\big)dx+\int_0^T\int_0^1\theta^\beta\theta_xdxdt+\int_0^T\int_0^1 u_{xx}^2dxdt\leq C.
\end{equation}
Then, rewriting \eqref{NSFAC-Lagrange}$_2$ as
\begin{equation}\label{momentum equation}
  u_t=-\big(\frac{\theta}{v}\big)_x+\frac{u_{xx}}{v}-\frac{u_x v_x}{v^2}-\epsilon\frac{\phi_x}{v}\big(\frac{\phi_x}{v}\big)_x,
\end{equation}
combining with \eqref{bound of the square modulus of the first derivative}, \eqref{max theta^2}, \eqref{the upper and lower bounds of the derivative for phi}, \eqref{u-x L2}, \eqref{bound of the square modulus of the third derivative for phi} and \eqref{tempeture and velocity-1}, we achieve
\begin{eqnarray}\label{sup u-t L2}
 \int_0^T\int_0^1u_t^2dxdt&\leq& C\int_0^T\int_0^1\Big(u_{xx}^2+u_x^2v_x^2+\theta_x^2+\theta^2v_x^2+\big|\frac{\phi_x}{v}\big|^2\big|\big(\frac{\phi_x}{v}\big)_x\big|^2\Big)dxdt\notag\\
 &\leq& C,
\end{eqnarray}
 together with \eqref{tempeture and velocity-1}, the energy inequality \eqref{energy estimate of velocity} is obtained. The proof of Lemma \ref{velocity-x} is completed.
\end{proof}

\begin{lemma}\label{theta-x}
Let $(v,u,\theta,\phi)$ be a smooth solution of \eqref{NSFAC-Lagrange}-\eqref{initial condition} on $[0,1]\times [0,T]$. Then it holds that for $\forall(x,t)\in[0,1]\times [0,T]$
 \begin{equation}\label{energy estimate of tempeture}
\sup_{0\leq t\leq T}\int_0^1\theta_{x}^2dx+\int_0^T\int_0^1\big(\theta_{t}^2+\theta^2_{xx}\big)dx\leq C.
 \end{equation}
\end{lemma}
\begin{proof}
Multiplying \eqref{NSFAC-Lagrange}$_5$ by $\theta^\beta\theta_t$ and integrating the resultant over $(0,1)$, by using \eqref{bound of density}, \eqref{bound of phi}, \eqref{tempeture and velocity-1}, we have
\begin{align}\label{theta-t theta-xx}
 &\frac{1}{2}\frac{d}{dt}\Big(\int_0^1\frac{(\theta^\beta\theta_x)^2}{v}dx\Big)+\int_0^1\theta^\beta\theta_t^2dx\notag \\
  & =-\frac{1}{2}\int_0^1\frac{(\theta^\beta\theta_x)^2u_x}{v^2}dx+\int_0^1\frac{\theta^\beta\theta_t\big(-\theta u_x+u_x^2+v^2\mu^2\big)}{v}dx\notag\\
  &\leq C\max_{x\in[0,1]}|u_x|\theta^{\frac{\beta}{2}}\int_0^1\theta^{\frac{3\beta}{2}}\theta_x^2dx+\frac{1}{2}\int_0^1\theta^\beta\theta_t^2dx+C\int_0^1\theta^{\beta+2}u_x^2dx+C\int_0^1\theta^\beta\big(u_x^4+\mu^4\big)dx\notag\\
  &\leq C\int_0^1\theta^{\beta}\theta_x^2dx\int_0^1(\theta^{\beta}\theta_x)^2dx+\frac{1}{2}\int_0^1\theta^{\beta}\theta_t^2dx+C\max_{x\in[0,1]}\big(\theta^{2\beta+2}+u_x^4+\mu^4\big)+C.
\end{align}
Now we deal with the term $\max_{x\in[0,1]}\big(\theta^{2\beta+2}+u_x^4+\mu^4\big)$ in the last inequality of \eqref{theta-t theta-xx}. Combining Lemma \ref{velocity-x}, direct computation shows that
\begin{eqnarray}
  &&\int_0^T\max_{x\in[0,1]}u_x^4dt\notag\\
  &\leq&C\int_0^T\int_0^1u_x^4dxdt+C\int_0^T\int_0^1|u_x^3u_{xx}|dxdt\notag\\
  &\leq&C\int_0^T\max_{x\in[0,1]}u_x^2\int_0^1u_x^2dxdt+C\int_0^T\max_{x\in[0,1]}u_x^2\Big(\int_0^1u_x^2dx\Big)^\frac{1}{2}\Big(\int_0^1u_{xx}^2dx\Big)^\frac{1}{2}dt\notag\\
   &\leq&\frac12\int_0^T\max_{x\in[0,1]}u_x^4dt+C\int_0^T\int_0^1\big(u_x^2+u_{xx}^2)dxdt\notag\\
      &\leq&\frac12\int_0^T\max_{x\in[0,1]}u_x^4dt+C,
\end{eqnarray}
so we get immediately that
\begin{equation}\label{u-x^4}
 \int_0^T\max_{x\in[0,1]}u_x^4dt\leq C.
\end{equation}
By the same way above, combining with \eqref{bound of the square modulus of the third derivative for phi}, we  have
\begin{equation}\label{mu^4}
 \int_0^T\max_{x\in[0,1]}\mu^4dt\leq C.
\end{equation}
Moreover, by using Sobolev embedding theorem, we get
\begin{equation}\label{the upper bound estimate for theta}
  \max_{x\in[0,1]}\theta^{2\beta+2}\leq C+C\int_0^1(\theta^\beta\theta_x)^2dx,
\end{equation}
Substituting \eqref{u-x^4}, \eqref{mu^4}, \eqref{the upper bound estimate for theta} into \eqref{theta-t theta-xx}, by using Gronwall's inequality and \eqref{the upper bounds of density}, we obtain
\begin{equation}\label{derivative estimation for theta}
\sup_{0\leq t\leq T}\int_0^1(\theta^{\beta}\theta_x)^2dx+\int_0^T\int_0^1\theta^\beta\theta_t^2dxdt\leq C.
\end{equation}
Therefore,  in view of \eqref{the upper bound estimate for theta}, we achieve
\begin{equation}\label{the upper bound estimate for temperature}
 \max_{(x,t)\in[0,1]\times[0,T]}\theta\leq C.
\end{equation}
Thus, both \eqref{the upper bound estimate for theta} and \eqref{derivative estimation for theta} lead to
\begin{align}
\displaystyle  \sup_{0 \le t\le T}\int_0^1\theta_{x}^2 dx+\int_0^T\int_0^1 \theta_t^2dxdt\le C.
\end{align}

Let's go back to \eqref{NSFAC-Lagrange} again, we have
\begin{equation}\label{energy conservation equation theta-xx}
\frac{\theta^\beta\theta_{xx}}{v}=\theta_t- \frac{\beta\theta^{\beta-1}\theta_x^2}{v}+\frac{\theta^\beta\theta_x v_x}{v^2}+\frac{\theta u_x}{v}-\frac{u_x^2+(v\mu)^2}{v},
\end{equation}
which yields that
\begin{align}\label{the estimate of tempture-xx L2}
\int_0^T\int_0^1\theta_{xx}^2dxdt&\leq C \int_0^T\int_0^1\big(\theta_x^4+\theta_x^2v_x^2+u_x^4+\mu^4+u_x^2+\theta_t^2\big)dxdt\notag\\
 &\leq C+C\int_0^T\max_{x\in[0,1]}\theta_x^2dt\notag\\
  &\leq C+C\int_0^T\int_0^1\theta_{xx}^2dxdt.
\end{align}
Meanwhile, by using maximum principle, $-1\leq\phi\leq1$.
The proof of Lemma \ref{theta-x} is completed.
\end{proof}

\textbf{Proof of Theorem 1.1.}  By Lemma 2.1, there exists a $T_*>0$, such that the system \eqref{NSFAC-Lagrange}-\eqref{initial condition} has a  unique strong solution $(v,u,\theta,\phi)$ on $(0,T_*]$ satisfying \eqref{local solution space}. Suppose that  $T_0$ is the maximum existence time of the unique strong solution $(v,u,\theta,\phi)$ to \eqref{NSFAC-Lagrange}-\eqref{initial condition}. Therefore, $T_0\geq T_*>0.$ We claim that
\begin{equation}\label{T0}
  T_0=+\infty.
\end{equation}
If not ,  $T_0<+\infty$, then by using  Lemma 2.2--2.9, the global a priori estimates of the solutions ensure $ (v,u,\theta,\phi)\in C([0,T_0];H^1)$ and that
\begin{equation}\label{upper bound}
  \|(v,u,\theta,\phi)(t)\|_{H^1}\leq C(T_0)<+\infty,\ \ \forall t\in[0,T_0],
\end{equation}
 where    $C(T_0) $   is a  positive constant  depending only on $T_0$,
 $\inf\limits_{x\in [0,1]}v_0(x)$, $ \inf\limits_{x\in [0,1]}\theta_0(x)$, and $\|(v_0 ,u_0,\theta_0 ,\phi_0)\|_{H^1(0,1)}$.  Thus, $(v,u,\theta,\phi)(x,T_0)$ is finite and well-defined.
 Thus, follows from  Lemma 2.1,   there exists a positive constant $T_{1}>0$  such that \eqref{NSFAC-Lagrange}-\eqref{initial condition} has a unique strong solution on $[T_0,T_0+T_1]$,
 which contradicts the definition of $T_0$. Thereby, the claim \eqref{T0} is true. The proof of Theorem 1.1 is completed.

\section*{Acknowledgement}
This research was partially supported by the National Natural Science Foundation
of China, No. 11901025 (the first author), No. 11971020(the second author) and No.
11671027(the fourth author).

\end{document}